\documentclass[11pt]{article}
\usepackage[a4paper,margin=2cm]{geometry}
\usepackage{graphicx}
\usepackage{color}
\usepackage{amsmath}
\usepackage{amssymb}
\usepackage{amsthm}
\usepackage{bbm}
\usepackage{framed}
\newtheorem{lem}{Lemma}
\newtheorem{thm}{Theorem}
\newtheorem{conj}{Conjecture}
\newtheorem{prop}{Proposition}

\def\diffd{\mathrm{d}}
\def\R{\mathbb R}

\def\uppar#1{^{{\scriptscriptstyle(}#1{\scriptscriptstyle)}}}

\title{A note of the convergence of the Fisher-KPP front centred around its
$\alpha$-level}
\author{Julien Berestycki\footnote{Department of Statistics and Magdalen College, University of Oxford, UK}, \'Eric Brunet\footnote{Laboratoire de Physique
Statistique, \'Ecole Normale
Sup\'erieure, PSL Research University; Universit\'e Paris Diderot Sorbonne
Paris-Cit\'e; Sorbonne Universit\'es UPMC Univ Paris 06; CNRS; 24 rue Lhomond,
75005 Paris, France.}}
\begin{document}
\maketitle

\begin{abstract}
We consider the solution $u(x,t)$ of the Fisher-KPP equation $\partial_t
u=\partial_x^2u+u-u^2$ centred around its $\alpha$-level
$\mu_t\uppar\alpha$ defined as $u(\mu_t\uppar\alpha,t)=\alpha$. It is well
known that for an initial datum that decreases fast enough, then
$u(\mu_t\uppar\alpha+x,t)$ converges as $t\to\infty$ to the critical
travelling wave. We study in this paper the speed of this convergence and
the asymptotic expansion of $\mu_t\uppar\alpha$ for large~$t$. It is known
from Bramson \cite{Bramson83} that for initial conditions that decay fast
enough, one has $\mu_t\uppar\alpha=2t-(3/2)\ln t+\text{Cste}+o(1)$. Work
is under way \cite{nrr} to show
that the $o(1)$ in the expansion is in fact
a $k\uppar\alpha/\sqrt t+\mathcal O(t^{\epsilon-1})$ for any $\epsilon>0$
for some $k\uppar\alpha$, where it is not clear at this point whether
$k\uppar\alpha$ depends or not on $\alpha$. We show that, unless
the time derivative of $\mu_t\uppar\alpha$ has a very unexpected behaviour
at infinity, the
coefficient $k\uppar\alpha$ does not, in fact, depend on $\alpha$.

We also conjecture that, for an initial condition that decays fast enough,
one has in fact $\mu_t\uppar\alpha=2t-(3/2)\ln
t+\text{Cste}-(3\sqrt\pi)/\sqrt t+g (\ln t)/t +o (1/t)$ for some
constant~$g$ which does not depend on $\alpha$.

\end{abstract}

\section{Introduction}
We consider the Fisher-KPP equation \cite{Fisher,KPP} which describes the evolution of
$(x,t)\mapsto  u(x,t)$:
\begin{equation}
\partial_t u = \partial_x^2 u +u-u^2,\qquad u(x,0)=u_0(x).
\label{FKPP}
\end{equation}

Bramson \cite{Bramson83} proved that if the initial condition $u_0(x)$ is
such that
\begin{equation}
\begin{cases}
0\le u_0(x)\le 1,\\
\displaystyle
\limsup_{x\to \infty} \frac1x \ln\bigg[  \int_x^{x(1+h)} u_0(y) \,\diffd
y   \bigg]\le -1&\text{for some (all) $h>0$},\\[3ex]
\displaystyle
\lim_{x\to-\infty}u_0(x)=1,
\end{cases}
\label{cond}
\end{equation}
then the shape of the front around an appropriately chosen centring term
$m_t$ converges to the critical wave $\omega(x)$:
\begin{equation}
u(m_t+x,t)\xrightarrow[t\to\infty]{}\omega(x)\quad\text{uniformly in $x$,}
\label{unif}
\end{equation}
where $\omega(x)$ is the unique
solution to
\begin{equation}
0=\omega'' + 2 \omega' +\omega-\omega^2,\qquad
\omega(-\infty)=1,\quad\omega(+\infty)=0,\quad\omega(0)=\frac12.
\label{omega}
\end{equation}
(The second line of \eqref{cond} means that $u_0(x)$ decays roughly as fast
or faster than $e^{-x}$. The third line could be weakened
considerably.)

Furthermore, if and only if $u_0(x)$ satisfies the stronger condition
\begin{equation}
\int\diffd x\,x e^x u_0(x)<\infty,
\label{stronger}
\end{equation}
then any valid centring term $m_t$ in the sense of \eqref{unif} must be of
the form
\begin{equation}
m_t= 2t -\frac32\ln t + C + o(1),
\label{Bram1}
\end{equation}
where the constant $C$ depends on the initial condition $u_0(x)$.

If \eqref{cond} holds then,
as shown in Section~\ref{sec:Proofs},
for each $\alpha$ and for each time $t$ if it
is large enough,
there exists a unique
$\mu_t\uppar\alpha$ such that
\begin{equation}
u(\mu_t\uppar\alpha,t)=\alpha.
\end{equation}
Furthermore, $t\mapsto\mu_t\uppar\alpha$ is $\mathcal C^1$ for $t$ large
enough.
Introducing $W\uppar\alpha$ as the unique antecedent of $\alpha$ by $\omega$,
\begin{equation}
\omega(W\uppar\alpha)=\alpha,
\end{equation}
it is then easy to see that
 $\mu_t\uppar\alpha-W\uppar\alpha$ must be a valid
choice for $m_t$ and that one has
\begin{equation}
u(\mu_t\uppar\alpha+x,t)\xrightarrow[t\to\infty]{}\omega(W\uppar\alpha+x)
\qquad\text{uniformly in $x$},
\end{equation}

In particular, if \eqref{stronger} holds
\begin{equation}
\mu_t\uppar\alpha=2t-\frac32\ln t + C +W\uppar\alpha+o(1).
\label{Bram}
\end{equation}
where $C$ is the constant from \eqref{Bram1} and, as such, depends on the
initial condition but not on $\alpha$.

It makes sense to try to determine the next term in the large $t$ expansion
of $\mu_t\uppar\alpha$. A famous conjecture \cite{evs} states that
\begin{equation}
\text{If $u_0(x)$ decays ``fast enough''},\qquad
\mu_t\uppar\alpha=2t-\frac32\ln t + C +W\uppar\alpha-\frac{3\sqrt\pi}{\sqrt
t}+o(t^{-1/2}),
\label{EvS}
\end{equation}
where they claim that, remarkably, the coefficient $-3\sqrt\pi$ of the
$t^{-1/2}$ term depends neither on $\alpha$ nor on the initial condition.
Two recent works \cite{Brunet2015,BBHR2015}
looking at linearised versions of the Fisher-KPP recover
suggest that \eqref{EvS} might 
only hold if $\int\diffd x\,x^2 u_0(x)<\infty$ (compare to the
condition~\eqref{stronger} under which \eqref{Bram} holds) ; in particular, if
$u_0(x)\sim A x^{\kappa}e^{-x}$ with $-3\le\kappa<-2$, one would have
Bramson's logarithmic correction~\eqref{Bram} but the first vanishing
correction would be different from that in~\eqref{EvS}.

Work is under way \cite{nrr} to prove that, indeed, the first vanishing
term in the expansion of $\mu_t\uppar\alpha$ is of order $t^{-1/2}$, but,
as of now, the precise value of the coefficient and, crucially, whether or
not it depends on $\alpha$, is still an open question.

\medskip

The goal of this letter is not to prove \eqref{EvS}, but to put some
constraint on what the first vanishing correction might look like.
For instance, our Theorem~\ref{thm1} states
that for any initial condition such that \eqref{cond}
holds, and any values of $\alpha$ and $\beta$ such that 
$0<\alpha<\beta<1$, one has
\begin{equation}
\text{If $2-\dot\mu_t\uppar\alpha=\mathcal O(t^{-\gamma})$ for some
$\gamma>0$,}\qquad
\text{then}\quad\mu_t\uppar\alpha-\mu_t\uppar\beta = W\uppar\alpha-W\uppar\beta+
 \mathcal O(t^{-\gamma}),
\label{rough}
\end{equation}
where $\dot\mu_t\uppar\alpha$ is the derivative of $t\to\mu_t\uppar\alpha$.

A natural question is of course how large can $\gamma$ be chosen in
\eqref{rough}.
In the Physics literature, it is often
assumed, and without batting an eye, that when \eqref{Bram} holds, then
$2-\dot\mu_t\uppar\alpha$ must be equivalent to $3/(2t)$. Of course, one is
not allowed in general to differentiate asymptotic expansions but,
intuitively, if the initial condition 
decays fast enough at infinity, then the heat operator in the Fisher-KPP
equation~\eqref{FKPP} should smooth everything and the functions
$\mu_t\uppar\alpha$ should be extremely well behaved for large times.
It would then seem that ``$2-\dot\mu_t\uppar\alpha=\mathcal O(t^{-1})$'' is
a fair conjecture. To our knowledge, there is no rigorous result on this,
but it would imply the
 following Conjecture:

\begin{conj}
Pick $\alpha$ and $\beta$ in $(0,1)$.
For any initial condition $u_0(x)$ such that \eqref{cond} holds, then
\begin{equation}
\mu_t\uppar\alpha-\mu_t\uppar\beta = W\uppar\alpha-W\uppar\beta+\mathcal
O\Big(\frac1t\Big).
\label{conj1eq}
\end{equation}
\label{conj1}
\end{conj}
In Section~\ref{numev}, we present some numerical evidence in support of
Conjecture~\ref{conj1} for a step initial condition.
What~\eqref{conj1eq} basically means is that
any term larger than $1/t$  in a large $t$
asymptotic expansion of $\mu_t\uppar\alpha$ must have a coefficient independent
of $\alpha$. In particular:
\begin{itemize}
\item If, as claimed in \cite{evs,nrr}, the first vanishing correction in
$\mu_t\uppar\alpha$ is of order $1/\sqrt t$ for initial conditions that
decay fast enough, then its coefficient must be
independent of $\alpha$.
\item The results of \cite{Brunet2015,BBHR2015} suggest that if the initial
condition is asymptotically of the form $u_0(x)\sim A x^\kappa e^{-x}$ with
$\kappa\in(-3,-2)$, then the first vanishing correction should be of order 
$t^{1+\frac\kappa2}$. If it is indeed the case, our conjecture implies that
the coefficient is independent of $\alpha$.
\end{itemize}
If Conjecture~\ref{conj1} turns out to be incorrect and if, for instance,
for some
initial condition $u_0(x)$, one has
$\mu_t\uppar\alpha=2t-\frac32\ln t+C+W\uppar\alpha
+k\uppar\alpha t^{-1/2}+\mathcal O(t^{-0.99})$ with a coefficient $k\uppar\alpha$
which is not a constant function of $\alpha$, then Theorem~\ref{thm1} below
implies that $2-\dot\mu_t\uppar\alpha$ is not a $\mathcal O(1/t)$ and
Theorem~\ref{thm3}
below implies that $\dot\mu_t\uppar\alpha$ oscillates around 2 at infinity,
which would be quite unexpected.

\medskip

In Section~\ref{further}, we present some work on a solvable model in the
Fisher-KPP class which was introduced in \cite{Brunet2015}. This leads
us to another Conjecture on the asymptotic expansion of $\mu_t\uppar\alpha$
\begin{conj}
\label{conj2}
For an initial condition $0\le u_0(x)\le 1$ with $\lim_{x\to-\infty}
u_0(x)=1$ and 
\begin{equation}
\int u_0(x) x^3 e^x\,\diffd x <\infty,
\end{equation}
one has
\begin{equation}
\mu_t\uppar\alpha=2t-\frac32\ln t + C +W\uppar\alpha-\frac{3\sqrt\pi}{\sqrt
t}+g \frac{\ln t}t + \mathcal O   \Big(\frac1t\Big),
\end{equation}
for some constant $g$ which, by
Conjecture~\ref{conj1}, does not depend on
$\alpha$. 
\end{conj}
The work presented in Section~\ref{further} suggests also that, maybe,
$g=\frac98(5-6\ln 2)\approx0.946$, and this value is compatible with
numerical simulations. However, this value for $g$ relies on transposing
by analogy a result derived on a front equation which is quite
different from the Fisher-KPP equation, and it remains of a very
speculative nature.

\section{Results}
We restrict ourselves to initial conditions $u_0(x)$ such that
\eqref{cond} holds.
Pick $\alpha\in(0,1)$ and introduce
\begin{equation}
\eta_t = 2 -\dot \mu_t\uppar\alpha.
\label{defeta}
\end{equation}
Implicitly, $\eta_t$ depends on $\alpha$.
One has, for large time  \cite{KPP,uchiyama},
\begin{equation}
\eta_t\to 0.
\label{eta0}
\end{equation}

\begin{thm}
\label{thm1}
Pick $\alpha\in(0,1)$. For any initial condition $u_0(x)$ such that
\eqref{cond} holds, if
\begin{equation}
\eta_t:=2-\dot\mu_t\uppar\alpha=\mathcal O(t^{-\gamma}) \text{  for some $\gamma>0$,}
\label{14}
\end{equation}
then, for any $x_0>0$,
\begin{equation}
\label{thm1_1}
\max_{x\in[-x_0,0]
}\Big|u(\mu_t\uppar\alpha+x,t)-\omega(W\uppar\alpha+x)\Big|=\mathcal
O(t^{-\gamma}),
\end{equation}
which implies that for any $\beta\in(\alpha,1)$,
\begin{equation}
\mu\uppar\alpha_t-\mu\uppar\beta_t = W\uppar\alpha-W\uppar\beta + \mathcal
O(t^{-\gamma}).
\label{thm1_2}
\end{equation}
If, furthermore, $\alpha>\frac12$, then the ``$\max_{x\in[-x_0,0]}$'' in
\eqref{thm1_1} can be replaced by a ``$\max_{x\le0}$''.
\end{thm}

\begin{thm}
\label{thm2}
Pick $\alpha\in(0,1)$. For any initial condition $u_0(x)$ such that
\eqref{cond} holds, if
\begin{equation}
\eta_t:=2-\dot\mu_t\uppar\alpha\ne0\text{ for $t$ large enough}\qquad
\text{and}\qquad
\frac{\dot\eta_t}{\eta_t}\to0,
\label{17}
\end{equation}
then, for any $x_0>0$,
\begin{equation}
\max_{x\in[-x_0,0]}\left|\frac{u(\mu_t\uppar\alpha+x,t)-\omega(W\uppar\alpha+x)}
{\eta_t}
-
\left(\Phi(W\uppar\alpha +x)
-\frac{\Phi(W\uppar\alpha)}{\omega'(W\uppar\alpha)}
\omega'(W\uppar\alpha+x)\right)
\right|
\xrightarrow[t\to\infty]{}0,
\label{thm2_1}
\end{equation}
where $\Phi$ is
\begin{equation}
\label{defPhi}
\Phi(x)=\omega'(x)\int_0^x\diffd y\,
\frac{e^{-2y}}{\omega'(y)^2}\int_{-\infty}^y\kern-1em\diffd
z\,\omega'(z)^2e^{2z}.
\end{equation}
This implies that for any $\beta\in(\alpha,1)$,
\begin{equation}
\mu\uppar\alpha_t-\mu\uppar\beta_t = W\uppar\alpha-W\uppar\beta
- \eta_t\left[
\frac{\Phi(W\uppar\alpha)}{\omega'(W\uppar\alpha)}
-
\frac{\Phi(W\uppar\beta)}{\omega'(W\uppar\beta)}+o(1)\right].
\label{thm2_2}
\end{equation}
If, furthermore, $\alpha>\frac12$, then the ``$\max_{x\in[-x_0,0]}$'' in
\eqref{thm2_1} can be replaced by a ``$\max_{x\le0}$''.
\end{thm}

\begin{thm}
\label{thm3}
Pick $\alpha\in(0,1)$. For any initial condition $u_0(x)$ such that
\eqref{cond} holds, if
\begin{equation}
\begin{cases}
\displaystyle
\mu_t\uppar\alpha = 2t -\frac32\ln t + C + W\uppar\alpha - \frac{g}{\sqrt t} + 
 \mathcal O(t^{-\gamma})\quad\text{for some $\gamma\in\big(\frac12,1\big]$,}
\\[1ex]
\displaystyle
\text{$\eta_t:=2-\dot\mu_t\uppar\alpha$ has a constant sign for $t$ large enough},
\end{cases}
\label{propmualpha}
\end{equation}
then, for any $x_0>0$,
\begin{equation}
\max_{x\in[-x_0,0]
}\Big|u(\mu_t\uppar\alpha+x,t)-\omega(W\uppar\alpha+x)\Big|=\mathcal
O(t^{-\gamma}),
\label{thm3_1}
\end{equation}
which implies that for any $\beta\in(\alpha,1)$,
\begin{equation}
\mu_t\uppar\beta = 2t -\frac32\ln t + C + W\uppar\beta
- \frac{g}{\sqrt t} + \mathcal O(t^{-\gamma}),
\label{thm3_2}
\end{equation}
where we emphasize that the coefficient $g$ is the same as in
\eqref{propmualpha}.

If, furthermore, $\alpha>\frac12$, then the ``$\max_{x\in[-x_0,0]}$'' in
\eqref{thm3_1} can be replaced by a ``$\max_{x\le0}$''.
\end{thm}

\noindent\textit{Remarks:}
\begin{itemize}
\item Theorem~\ref{thm2} is more precise than Theorem~\ref{thm1}, but
requires to make
some assumptions on the second derivative on $\mu_t\uppar\alpha$.
\item In Theorem~\ref{thm2}, if $2-\dot\mu_t\uppar\beta$ satisfies the same
hypothesis as $\eta_t=2-\dot\mu_t\uppar\alpha$, then it is easy to see that
necessarily
$2-\dot\mu_t\uppar\beta\sim2-\dot\mu_t\uppar\alpha$.
\item In Theorem~\ref{thm2}, one checks that $\Phi$ is
 the unique solution to
\begin{equation}
\Phi''+2\Phi'+(1-2\omega)\Phi
=\omega'(x),\qquad
\Phi(0)=0,\qquad\Phi(-\infty)=0.
\end{equation}
\item The results above concern only convergence for negative
$x$. However, in each Theorem, we could replace the
``$\max_{x\in[-x_0,0]}$''
by a ``$\max_{x\le x_0}$'' if one assume that the hypothesis on
$\mu_t\uppar\alpha$ does not hold only for the one value of $\alpha$ that
we pick, but holds in fact for all values of $\alpha$ (as in ``there exists
a $\gamma$ such that \eqref{14} or \eqref{propmualpha} hold for all
$\alpha$'', or ``\eqref{17} holds for all $\alpha$'').
Indeed, one would
simply have to apply the Theorems as written above once for an $\alpha'$ small
enough to encompass what happens at $x=x_0$, another time for an $\alpha''$
larger than $1/2$, and then glue together the results.
\item The theorems do not assume that $u_0(x)$ is such that we are in the
regime~\eqref{stronger} with the $-\frac32\ln t$ of Bramson. It merely assumes that the
critical travelling wave $\omega$ is reached.
\item With Theorem~\ref{thm1}, it would be sufficient to prove that
$2-\dot\mu_t\uppar\alpha=\mathcal O(t^{-1})$ for any $\alpha\in(0,1)$
to obtain Conjecture~\ref{conj1}.
\end{itemize}

\section{Proofs}
\label{sec:Proofs}

We start by proving the following result which was mentioned in the
introduction 
\begin{lem}
Suppose that the initial condition $u_0(x)$ is such that \eqref{cond} holds and
fix $\alpha \in (0,1)$. Then, for $t$ large enough, $\mu_t\uppar\alpha$ is
the unique solution of $u(x,t)=\alpha$ and furthermore $ t\mapsto
\mu_t\uppar\alpha$ is differentiable.
\end{lem}
\begin{proof}
Recall \eqref{unif}: there exists $m_t$ such that,
\begin{equation}
u(m_t+x,t)\to\omega(x)\quad\text{uniformly in $x$}.
\label{unif2}
\end{equation}
A standard result (see for instance \cite[Theorem~9.1]{uchiyama}) gives
then that
\begin{equation}
\partial_x u(m_t+x,t) \to\omega'(x)\quad\text{locally uniformly in $x$}.
\label{unif3}
\end{equation}
For any $t>0$ the function $x\mapsto u(x,t)$ is continuous and interpolates
between 1 and 0 so for each $t$ there exists at least one $x$ such that
$u(m_t+x,t)=\alpha$.
For each $\epsilon>0$, if
time is large enough, then $u(m_t+x,t)=\alpha$ implies that
$|x-W\uppar\alpha|\le\epsilon$
because of the uniform convergence \eqref{unif2}. As $\omega'(x)$ is negative and bounded
away from 0 on $[W\uppar\alpha-\epsilon,W\uppar\alpha+\epsilon]$ then
$\partial_xu(m_t+x,t)$ is negative on the same
interval for $t$ large enough because of \eqref{unif3}. This implies that for $t$ large enough
there exists a unique $x$ such that $u(m_t+x,t)=\alpha$ or, equivalently,
a unique $\mu_t\uppar\alpha$ such that $u(\mu_t\uppar\alpha,t)=\alpha$.
The differentiability of $t\mapsto \mu_t\uppar\alpha$
is then a consequence of the implicit function Theorem.
\end{proof}

We now turn to the proofs of the Theorems.
Pick $\alpha\in(0,1)$ and an initial condition $u_0(x)$ satisfying \eqref{cond}.
Introduce
\begin{equation}
\tilde\omega(x)=\omega(W\uppar\alpha+x).
\end{equation}
When $t$ is sufficiently large so that
$t\mapsto\mu_t\uppar\alpha$ is a well-defined $\mathcal C^1$ function, introduce
also
\begin{equation}
	\delta(x,t) = u(\mu_t\uppar\alpha+x,t)-\tilde\omega(x).
\label{defdelta}
\end{equation}
Of course,
\begin{equation}
|\delta(x,t)|\le1,\qquad
\delta(0,t)=0,\qquad
\delta(x,t)\xrightarrow[t\to\infty]{}0,\qquad\text{uniformly in $x$.}
\end{equation}
From \eqref{defdelta},
\begin{align}
\partial_t\delta &= \partial_x^2(\delta +\tilde\omega) + \dot
\mu\uppar\alpha_t
\partial_x(\delta+\tilde\omega)+(\delta+\tilde\omega)
-(\delta+\tilde\omega)^2,\notag
\\&=\partial_x^2\delta +\dot \mu\uppar\alpha_t
\partial_x\delta+(1-2\tilde\omega)\delta-\delta^2-(2-\dot
\mu\uppar\alpha_t)\tilde\omega',\notag
\\&
=\partial_x^2\delta+2\partial_x\delta-(2\tilde\omega-1+\delta)\delta
-\eta_t\tilde\omega' -\eta_t\partial_x\delta,
\label{maindelta}
\end{align}
where we used \eqref{omega} to simplify the $\tilde\omega$ and where we recall \eqref{defeta}:
\begin{equation}
\eta_t=2-\dot \mu\uppar\alpha_t.
\label{defeta2}
\end{equation}
Define $r$ by
\begin{equation}
r(x,t)=e^x\delta(x,t).
\end{equation}
One finds that $r$ satisfies for all $x\in\R$ 
\begin{equation}\label{mainr}
\partial_t  r = \partial_x^2  r -(2\tilde\omega+\delta) r
+ (r-\tilde\omega'e^x)\eta_t- \eta_t \partial_x  r, \qquad
 r(0,t)= 0.
\end{equation}
The condition $r(0,t)=0$ effectively decouples the domains $x\le0$ and
$x\ge0$. We can therefore consider~\eqref{mainr} for $x\le0$ only.
A key step in our proofs is the following:
\begin{prop}
With $u_0(x)$ satisfying~\eqref{cond},
there exists two positive constants $c$ and $t_0$
such that
\begin{equation}
\max_{x\le0} \big|r(x,t)\big| \le e^{-\alpha t}\bigg(e^{\alpha t_0}+c\int_{t_0}^t\diffd
u |\eta_u| e^{\alpha u}\bigg)\quad\text{for all $t\ge t_0$}.
\label{eqprop}
\end{equation}
Furthermore, if $\alpha>1/2$, there exists two other positive constants $c$
and $t_0$ such that
\begin{equation}
\max_{x\le0} \big|\delta(x,t)\big| \le e^{-(\alpha-\frac12) t}\bigg(e^{
(\alpha-\frac12) t_0}+c\int_{t_0}^t\diffd
u |\eta_u| e^{(\alpha-\frac12) u}\bigg)\quad\text{for all $t\ge t_0$}.
\label{eqprop2}
\end{equation}
\label{prop}
\end{prop}
The right-hand-sides in \eqref{eqprop} and \eqref{eqprop2} can then be estimated with the
following Lemma:
\begin{lem}
\label{lem}
For $\beta>0$ and $t_0$ two real numbers, and $t\mapsto\phi_t$ a function,
define $R_t$ by
\begin{equation}
R_t= e^{-\beta t}\int_{t_0}^t\diffd u\,\varphi_ue^{\beta
u} .
\label{lemR=}
\end{equation}
For large time,
\begin{itemize}
\item If $\varphi_t\to0$, then $R_t\to 0$,
\item If $\varphi_t=\mathcal O(t^{-\gamma})$ for some $\gamma>0$,
then $R_t=\mathcal O(t^{-\gamma})$,
\item If $\int_t^\infty \varphi_u\,\diffd u= \mathcal O(t^{-\gamma})$ for some
$\gamma>0$,
then $R_t=\mathcal O(t^{-\gamma})$.
\end{itemize}
\end{lem}

With Proposition~\ref{prop} and Lemma~\ref{lem}, the first part of
Theorem~\ref{thm1} is trivial: assuming that $\eta_t=\mathcal
O(t^{-\gamma})$ with $\gamma>0$, then 
\begin{equation}
\begin{cases}
\displaystyle
\max_{x\in[-x_0,0]}|\delta(x,t)|\le
e^{x_0}\max_{x\le0}|r(x,t)|=\mathcal
O(t^{-\gamma}),
\\[2ex]\displaystyle
\max_{x\le0}|\delta(x,t)|
=\mathcal O(t^{-\gamma}),&\text{if $\alpha>\dfrac12$}.
\end{cases}
\label{concthm1thm3}
\end{equation}
which is \eqref{thm1_1} of Theorem~\ref{thm1}.

The first part of Theorem~\ref{thm3} is also very
easy. Assuming \eqref{propmualpha}, then one has, for some
$\gamma\in(\frac12,1]$,
\begin{equation}
\eta_t=\frac3{2t}-\frac g{2t^{3/2}}+\psi_t\qquad\text{with
}\int_t^\infty\diffd u\,\psi_u=\mathcal O(t^{-\gamma}).
\label{thm3etat}
\end{equation}
Because we assume that $\eta_t$ does not change sign for $t$ large
enough, one can push the absolute values around $\eta_t$ in
\eqref{eqprop} and \eqref{eqprop2} outside the integral. Then, applying
Lemma~\ref{lem} to each of
the three terms composing $\eta_t$ in \eqref{thm3etat}, one reaches again
the conclusion \eqref{concthm1thm3} which is \eqref{thm3_1} in
Theorem~\ref{thm3}.

The second parts of Theorems~\ref{thm1} and~\ref{thm3} are then direct
consequences of their first parts.
Pick $\beta\in(\alpha,1)$. By definition,
\begin{equation}
\beta=u\big(\mu\uppar\alpha_t+(\mu_t\uppar\beta-\mu_t\uppar\alpha),t\big)
=\omega\big(W\uppar\alpha
+\mu_t\uppar\beta-\mu_t\uppar\alpha\big)
+\delta\big(\mu_t\uppar\beta-\mu\uppar\alpha_t,t\big).
\label{b=}
\end{equation}
We know that $\mu_t\uppar\beta-\mu_t\uppar\alpha$ converges to
$W\uppar\beta-W\uppar\alpha<0$, so it must remain inside $[-x_0,0]$ for $t$
large enough and a well chosen $x_0$. Then, the term $\delta(\cdot,t)$ in
\eqref{b=} is
a $\mathcal O(t^{-\gamma})$ and because $\omega$ is differentiable 
with non-zero derivative, one must have
\begin{equation}
\mu\uppar\beta_t-\mu\uppar\alpha_t = W\uppar\beta-W\uppar\alpha + \mathcal
O(t^{-\gamma}),
\label{thm1_2bis}
\end{equation}
which is the second part \eqref{thm1_2} of Theorem~\ref{thm1}.
The conclusion \eqref{thm1_2bis} also holds for Theorem~\ref{thm3};
combined with its hypothesis~\eqref{propmualpha}, it
gives the second part~\eqref{thm3_2} of Theorem~\ref{thm3}.

Therefore, it only remains to prove Proposition~\ref{prop} 
and Lemma~\ref{lem}
 to complete the proofs of Theorems~\ref{thm1}
and~\ref{thm3}.

\begin{proof}[Proof of Proposition~\ref{prop}]
Recall equation~\eqref{mainr} followed by $r$,
\begin{equation}\label{mainr2}
\partial_t  r = \partial_x^2  r -(2\tilde\omega+\delta)r 
+ (r-\tilde\omega'e^x)\eta_t- \eta_t \partial_x  r, \qquad
 r(0,t)= 0.
\end{equation}
We only consider the side $x\le0$.
Since $r(x,t)=e^x \delta(x,t)$ we have
\begin{equation}
|r(x,t)|\le 1\qquad\text{for all $t$ and all $x\le0$}.
\label{c1}
\end{equation}
Furthermore, $-\tilde\omega'(x)e^x>0$ is bounded for $x\le0$
so there exists a $c$ such that
\begin{equation}
|r(x,t) -\tilde\omega'(x)e^x | \le c\qquad\text{for all $t$ and all $x\le0$}.
\label{c2}
\end{equation}
Also, there exists a $t_0$ such that
\begin{equation}
2\tilde\omega(x)+\delta(x,t) > \alpha
\label{c3}
\qquad\text{for $t\ge t_0$ and $x\le0$.}
\end{equation}
Indeed, $\tilde\omega(x)\ge\alpha$ for $x\le0$, and
$\delta(x,t)$ converges uniformly to 0.

With these ingredients we are ready to apply the comparison principle. It
goes in two steps; first, because of \eqref{c1} and \eqref{c2} one has
for all $x\le 0$ and all $t\ge t_0$
\begin{equation}
r(x,t)\le \hat r(x,t)\quad\text{where}\quad
\partial_t\hat r =\partial_x^2\hat r -(2\tilde\omega+\delta)
\hat r +c |\eta_t| -\eta_t\partial_x\hat r,\quad
\hat r(x,t_0)=1,\quad \hat r(0,t)=0.
\end{equation}
Clearly, $\hat r$ cannot become negative. Then,
one gets with
\eqref{c3} that for any non-negative function
$b_t$
\begin{equation}
\hat r(x,t)\le \overline r(x,t)\quad\text{where}\quad
\partial_t\overline r =\partial_x^2\overline r -\alpha
\overline r +c |\eta_t| -\eta_t\partial_x\overline r,\quad
\overline r(x,t_0)=1,\quad\overline r(0,t)=b_t.
\end{equation}
We choose $b_t$ so that $\overline r$ remains $x$ independent, which
leads to $\partial_t\overline r = -\alpha\overline r+c|\eta_t|$ or
\begin{equation}
\overline r(\cdot,t) = e^{-\alpha t}\bigg(e^{\alpha t_0}+c\int_{t_0}^t\diffd
u\,|\eta_u|e^{\alpha u}\bigg).
\end{equation}
Similarly, one shows that $-r(x,t)\le\hat r(x,t)\le \overline r(\cdot,t)$,
which concludes the proof.

Finally, we prove the second part of Proposition~\ref{prop} in exactly the
same way than the first part, but starting from~\eqref{maindelta} instead of
\eqref{mainr}. As above, one first shows that $\delta\le\hat\delta$ where
$\hat\delta$ follows the same equation as $\delta$ but with the
$-\tilde\omega'\eta_t$ replaced by $c|\eta_t|$.
Then, $\hat\delta\le\overline\delta$ where we
replace
$-(2\tilde\omega-1+\delta)\delta$ by $-(\alpha-\frac12)\delta$. Indeed, for
all $x<0$, one has $2\tilde\omega(x)-1\ge2\alpha-1$ and for $t$ large
enough $|\delta|\le\alpha-\frac12$.
\end{proof}


\begin{proof}[Proof of Lemma~\ref{lem}]\ 
\begin{itemize}
\item
The first bullet point is easy. Assume $\varphi_t\to0$. For any $\epsilon>0$
pick $t_1>t_0$ such that $|\varphi_t|<\epsilon$ for $t\ge t_1$. Then
\begin{equation}
|R_t| \le e^{-\beta t} \bigg(\int_{t_0}^{t_1}\diffd
u\,|\varphi_u|e^{\beta u}\bigg) + \epsilon\, e^{-\beta t} \int_{t_1}^t\diffd
u\,e^{\beta u}\le \frac{2\epsilon}{\beta}\quad\text{for $t$ large
enough.}
\end{equation}
\item
For the second bullet point we prove a slightly more general result.
Let $t\mapsto\tilde\varphi_t$ be a function such that
\begin{equation}
\tilde\varphi_t>0,\qquad
\ln\tilde\varphi_t\text{ is convex for $t>t_0$},\qquad
\liminf_{t\to\infty}\frac{\ln\tilde\varphi_t}t>-\beta.
\label{condphi}
\end{equation}
By convexity, for $t>t_0$ and $u\in[t_0,t]$, one has
\begin{equation}
\ln\tilde\varphi_u\le
\frac{\ln\tilde\varphi_t-\ln\tilde\varphi_{t_0}}{t-t_0}(u-t_0)
+\ln\tilde\varphi_{t_0}
\quad\text{and then }
\tilde\varphi_ue^{\beta u} \le \tilde\varphi_{t_0}e^{\beta
t_0+\big[\beta+\frac{\ln\tilde\varphi_t-\ln\tilde\varphi_{t_0}}
{t-t_0}\big](u-t_0)}.
\end{equation}
Because of the last hypothesis on $\tilde\varphi$ in \eqref{condphi}, there
exists a $c>0$ such that the term in square brackets in the equation above
is larger than $c$ for $t$ large enough. Then, for $t$ large enough,
\begin{equation}
0\le
\int_{t_0}^t\diffd u \,\tilde\varphi_ue^{\beta u}
\le\frac{\tilde\varphi_{t}e^{\beta t}-\tilde\varphi_{t_0}e^{\beta t_0}}c
\qquad\text{and then}\quad
e^{-\beta t}\int_{t_0}^t\diffd u\,\tilde\varphi_ue^{\beta u} 
=\mathcal O(\tilde\varphi_t).
\end{equation}

If one assumes now that $\varphi_t=\mathcal O(\tilde\varphi_t)$ where
$\tilde\varphi_t$ satisfies~\eqref{condphi}, then we conclude that
$R_t=\mathcal O(\tilde\varphi_t)$. As the functions
$\tilde\varphi_t=t^{-\gamma}$ with $\gamma>0$ satisfy these conditions, we
have proved the second bullet point.
\item We finally turn to the third bullet point. Let
$\Phi_t=\int_t^\infty\diffd u\,\varphi_u$. By integration by parts
\begin{equation}
R_t= \Phi_{t_0}e^{-\beta (t-t_0)}-\Phi_t
+\beta e^{-\beta t} \int_{t_0}^t\diffd u\, \Phi_u e^{\beta u}
\end{equation}
If one assumes that $\Phi_t=\mathcal O(t^{-\gamma})$ for some $\gamma>0$,
then an application of the second bullet point gives the third bullet
point.
\end{itemize}
\end{proof}

We now turn to the proof of Theorem~\ref{thm2}.

\begin{proof}[Proof of Theorem~\ref{thm2}]
Write
\begin{equation}
r(x,t)=\eta_t\big[\Psi(x) + s(x,t)\big].
\label{defs}
\end{equation}
Then, by substituting into \eqref{mainr} and after division by $\eta_t$,
\begin{equation}
\frac{\dot\eta_t}{\eta_t}\big[\Psi+s\big]
+ \partial_t s = \Psi''+\partial_x^2s
-(2\tilde\omega + \delta-\eta_t)\big[\Psi+s\big]
-\tilde\omega'e^x-\eta_t\big[\Psi'+\partial_x s\big]
.
\label{eq58}
\end{equation}
We choose for $\Psi$ the unique solution to
\begin{equation}
\Psi''-2\tilde\omega\Psi=\tilde\omega'e^x,
\qquad\Psi(0)=0,\qquad\Psi(x)\text{ is bounded for $x<0$}.
\label{defPsi}
\end{equation}
Before going further, let us check that the solution to \eqref{defPsi}
exists and is unique. First notice that
\begin{equation}
(\tilde\omega' e^x)'' - 2\tilde\omega(\tilde\omega' e^x) = 0.
\label{eq60}
\end{equation}
This leads to look for a solution $\Psi$ of the form
\begin{equation}
\Psi(x)=\tilde\omega'(x)e^x\lambda(x).
\end{equation}
One obtains
\begin{equation}
(\tilde\omega'e^x)\lambda'' + 2
(\tilde\omega'e^x)'\lambda'=\tilde\omega'e^x
\quad\text{which is the same as}\quad
\frac{\diffd}{\diffd x}\Big[(\tilde\omega'e^x)^2\lambda'\Big]=
(\tilde\omega'e^x)^2.
\end{equation}
One sees from \eqref{eq60} that $\tilde\omega'(x)e^x\sim C e^{\sqrt2 x}$ 
as $x\to-\infty$  for some constant $C$.
Then
\begin{equation}
\lambda'(x)=\frac1{(\tilde\omega'e^x)^2}\left[A+\int_{-\infty}^x
\kern-.5em\diffd
z\,\tilde\omega'(z)^2e^{2z}\right]
\quad\text{and}\
\lambda(x)= B + \int_0^x\frac{\diffd y}{\tilde\omega'(y)^2e^{2y}}
\left[A+\int_{-\infty}^y\kern-.5em\diffd
z\,\tilde\omega'(z)^2e^{2z}\right].
\end{equation}
We take $B=0$ because we want $\Psi(0)=0$. Then one checks easily that one
must choose $A=0$ because otherwise $\Psi$ diverges at $-\infty$. Finally,
the only possible solution to \eqref{defPsi} is
\begin{align}
\Psi(x) &= \tilde\omega'(x)e^x
\int_0^x\frac{\diffd y}{\tilde\omega'(y)^2e^{2y}}
\int_{-\infty}^y\kern-.5em\diffd
z\,\tilde\omega'(z)^2e^{2z},
\\
&=\omega'(W\uppar\alpha+x)e^x\int_{W\uppar\alpha}^{W\uppar\alpha+x}
\frac{\diffd y}{\omega'(y)^2e^{2y}}
\int_{-\infty}^y\kern-.5em\diffd
z\,\omega'(z)^2e^{2z}
&&\text{(recall $\tilde\omega(x)=\omega(W\uppar\alpha+x)$)},
\\&= e^x\left[
\Phi(W\uppar\alpha +x)
-\frac{\Phi(W\uppar\alpha)}{\omega'(W\uppar\alpha)}
\omega'(W\uppar\alpha+x)
\right].
\label{PsiPhi}
\end{align}
with $\Phi$ the function \eqref{defPhi} defined in the Theorem.

We go back to \eqref{eq58}. Using \eqref{defPsi}, one gets
\begin{equation}
\partial_t s =\partial_x^2s -\Big(2\tilde\omega +\delta -\eta_t
+\frac{\dot\eta_t}{\eta_t}\Big)s -\eta_t\partial_x s
-\left[\frac{\dot\eta_t}{\eta_t}\Psi
+(\delta-\eta_t)\Psi +\eta_t\Psi'\right].
\label{eqs}
\end{equation}
$\Psi$ and $\Psi'$ are bounded for $x\le0$. For large time, $\delta(x,t)$
goes uniformly to zero. $\eta_t$ is known \cite{KPP,uchiyama}
to go to zero and, by hypothesis,
$\dot\eta_t/\eta_t$ also goes
to 0.  We conclude that there exists a positive 
function $t\mapsto\epsilon_t$ which
vanishes as $t\to\infty$ and such
that the term in square brackets in \eqref{eqs} lies for all $x\le0$ in
the interval $[-\epsilon_t,\epsilon_t]$.

The proof then goes as in Theorems~~\ref{thm1} and~\ref{thm3} by using in
two steps the comparison principle. For any $t_0$, for all $x\le0$ and all
$t\ge t_0$, one has
\begin{equation}
s(x,t)\le\hat s(x,t)\quad\text{where}\
\partial_t\hat s = \partial_x^2\hat s -
\Big(2\tilde\omega +\delta -\eta_t
+\frac{\dot\eta_t}{\eta_t}\Big)\hat s -\eta_t\partial_x \hat s +\epsilon_t,
\quad
\hat s(x,t_0)=c,\quad\hat s(0,t)=0,
\end{equation}
where $c$ is chosen such that $ s(x,t_0)\le c$ for all $x\le0$.
It is clear that $\hat s$ cannot become negative. Notice now, as before,
that the big parenthesis in the equation above is larger than $\alpha$ for
$t\ge t_0$ if $t_0$ is chosen large enough, uniformly in $x\le0$. 
Then, for any non-negative function $b_t$,
\begin{equation}
\hat s(x,t)\le\overline s(x,t)\quad\text{where}\
\partial_t\overline s = \partial_x^2\overline s -
\alpha
\overline s -\eta_t\partial_x \overline s +\epsilon_t,
\quad
\overline s(x,t_0)=c,\quad\overline s(0,t)=b_t,
\end{equation}
Choosing $b_t$ such that $\overline s$ is independent of $x$ and solving,
one obtains
\begin{equation}
s(x,t)\le c e^{-\alpha t}\bigg(e^{\alpha t_0}+\int_{t_0}^t\diffd
u\,\epsilon_u e^{\alpha u}\bigg).
\end{equation}
From Lemma~\ref{lem}, the right hand side goes to zero. One bounds $s(x,t)$
from below by a vanishing quantity in exactly the same way, therefore
\begin{equation}
\max_{x\le0} \big|s(x,t)\big|\xrightarrow[t\to\infty]{}0.
\end{equation}
Recalling that
\begin{equation}
u\big(\mu_t\uppar\alpha+x,t)-\omega(W\uppar\alpha+x)
=\delta(x,t)=e^{-x} r(x,t) = \eta_t\big[e^{-x}\Psi(x)+e^{-x}s(x,t)\big]
\end{equation}
and recalling the relation~\eqref{PsiPhi} between $\Psi$ and $\Phi$, this
gives the first part~\eqref{thm2_1} of Theorem~\ref{thm2}.

When $\alpha>\frac12$, one can go exactly through the same steps but
directly on $\delta(x,t)$: writing
\begin{equation}
\delta(x,t)=\eta_t\big[e^{-x}\Psi(x)+\tilde s(x,t)\big],
\end{equation}
then $\tilde s =e^{-x} s$ is solution to 
\begin{equation}
\partial_t \tilde s =\partial_x^2\tilde s +2\partial_x\tilde s
 -\Big(2\tilde\omega -1 +\delta +\frac{\dot\eta_t}{\eta_t}\Big)\tilde s
-\eta_t\partial_x \tilde s
-\left[\frac{\dot\eta_t}{\eta_t}\Psi
+(\delta-\eta_t)\Psi +\eta_t\Psi'\right]e^{-x}.
\end{equation}
One checks that the square bracket multiplied by $e^{-x}$ is still bounded,
then the parenthesis is larger than $\alpha-\frac12$ for $t$ large enough
and the comparison principle still applies and leads to $\max_{x\le0}|\tilde
s| \to 0$.

The second part~\eqref{thm2_2} of Theorem~\ref{thm2} is an easy consequence
of the first part; Apply \eqref{thm2_1} to
$x_t=\mu_t\uppar\beta-\mu_t\uppar\alpha$; as $x_t\to
W\uppar\beta-W\uppar\alpha$, it remains in $[-x_0,0]$ for $t$ large enough
and a well chosen $x_0$. One gets
\begin{equation}
\frac{\beta-\omega\big(W\uppar\alpha+\mu_t\uppar\beta-\mu_t\uppar\alpha\big)}
{\eta_t}
\xrightarrow[t\to\infty]{}\Phi\big(W\uppar\beta\big)-\frac{\Phi(W\uppar\alpha)}
{\omega'(W\uppar\alpha)}\omega'\big(W\uppar\beta\big).
\end{equation}
But
\begin{equation}
\omega\big(W\uppar\alpha+\mu_t\uppar\beta-\mu_t\uppar\alpha\big)
=\beta-\big(\mu_t\uppar\alpha
-\mu_t\uppar\beta
-W\uppar\alpha
+W\uppar\beta
\big)
\big[\omega'\big(W\uppar\beta\big)+o(1)\big],
\end{equation}
which leads to \eqref{thm2_2}.
\end{proof}

\section{Numerical evidence in support of Conjecture~\ref{conj1}}
\label{numev}

To better understand the behaviour of the $\mu_t\uppar\alpha$, we made some
numerical simulation. On a space-time lattice with steps $a$ and $b$,
we simulated the following equation:
\begin{equation}
h(x,t+b) = h(x,t) + \frac{b}{a^2}\Big[h(x-a,t)+h(x+a,t)-2h(x,t)\Big]
		+ b\Big[h(x,t)-h(x,t)^2\Big].
\label{eqlattice}
\end{equation}
We present here results for $a=0.1$ and $b=0.002$, but we also checked
other values of $a$ and $b$ and obtained similar results.

If one linearises \eqref{eqlattice} and looks for solutions of the form
$e^{-\gamma(x-vt)}$, one obtains the following relation
between $v$ and $\gamma$:
\begin{equation}
v(\gamma)=\frac1{\gamma b}\ln\Big[1+\frac b{a^2}\big(e^{\gamma a}+e^{-\gamma
a}-2\big)+b\Big],
\end{equation}
from which one computes the critical velocity $v_c=v(\gamma_c)$:
\begin{equation}
\text{For $a=0.1$ and $b=0.002$},\qquad 
     v_c=1.99684036732\ldots,\qquad
\gamma_c=1.00074727697\ldots
\label{vcgc}
\end{equation}
With $a$ and $b$ small, equation \eqref{eqlattice} is close in some sense
to the Fisher-KPP equation and the critical velocity and critical rate are close
to 2 and 1.

We simulated the front with a step initial condition.
It is expected that the relaxation of the front towards its critical
travelling wave is built from what happens in a region of size $2\sqrt t$
ahead of the front. It is therefore critical to have a good numerical
precision for the small values of $h$. For this reason, the data actually
stored in the computer's memory is $\ln h$ rather than $h$ itself.
On the left of the front, each time $\ln h$ was greater than $-10^{-16}$,
then $\ln h$ was set to 0. On the right of the front, the values of $h$
were computed only up to the position $v_ct + 10\sqrt t +50$; the values of
$h$ on the right of that boundary were approximated to be zero.
The simulation was run up to time 85\,000.

At each time-step, to measure $\mu_t\uppar\alpha$ with a sub-lattice
resolution, the computer looked at the four values of $\ln h$ which are
the closest to $\ln\alpha$ (two above $\ln\alpha$, and two below
$\ln\alpha$). From this four values, 
the interpolating polynomial of degree~3 was built, and the chosen value
for 
$\mu_t\uppar\alpha$ was the one for which this interpolating
polynomial gave $\ln\alpha$.

Figure~\ref{fig1} shows a graph of $\mu_t\uppar{1/2}-\mu_t\uppar\alpha$ as
a function of $1/t$ for different values of $\alpha$ for times larger than
$10^3$. On this scale, the data give some straight lines, suggesting that
$\mu_t\uppar{1/2}-\mu_t\uppar\alpha$ minus its large time limit is of order
$1/t$. This suggests strongly that Conjecture~\ref{conj1} holds for the
step initial condition and, therefore, that if there is a $1/\sqrt t$ term in the
asymptotic expansion of $\mu_t\uppar\alpha$, then the coefficient of this
term is $\alpha$-independent.

\begin{figure}[ht]
\centering
\includegraphics[width=.6\textwidth]{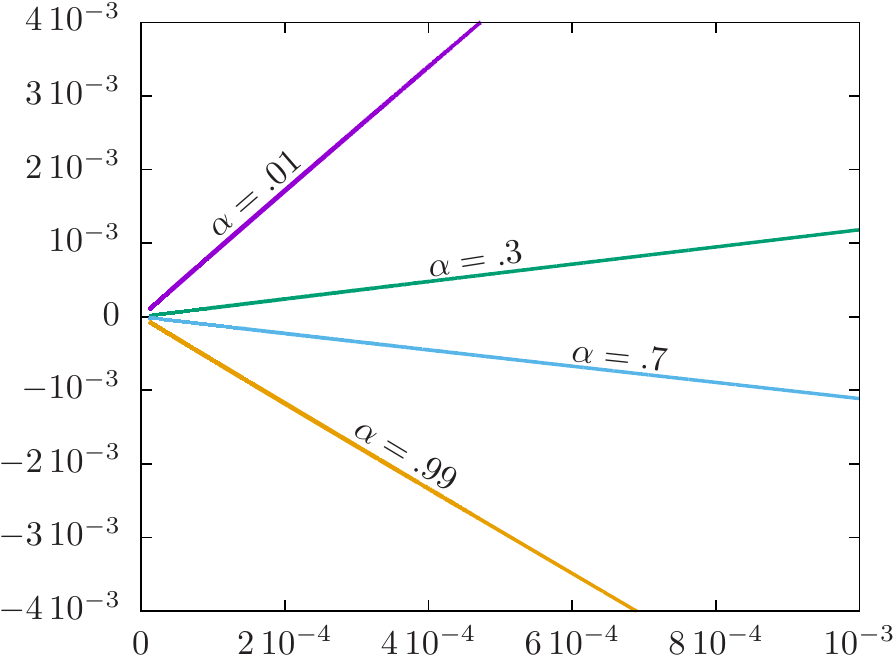}
\caption{$\mu_t\uppar{1/2}-\mu_t\uppar\alpha+\text{Cste}$ as a function of
$t^{-1}$, where the constant is chosen for each $\alpha$ so that the
curves meet at the origin.}
\label{fig1}
\end{figure}

\section{An exact expansion for a discrete solvable model}
\label{further}
In this section, we give some arguments in support
of Conjecture~\ref{conj2}. To that end,
we swap the Fisher-KPP equation we have been studying above for
a model
on a space lattice with
continuous time which was first introduced in
 \cite{Brunet2015} as a front in the universality class of the
Fisher-KPP equation: with $x\in\mathbb Z$ and $t\ge0$,
\begin{equation}
\partial_t u(x,t)= \begin{cases}
0 &\text{if $u(x,t)=1$}\\
u(x,t)+a u(x-1,t) & \text{if $u(x,t)<1$}.
\end{cases}\end{equation}
For this front, the function $v(\gamma)$ is given by
\begin{equation}
v(\gamma)=\frac1\gamma\big[1+e^\gamma\big],
\end{equation}
from which one obtains $v_c$ and $\gamma_c$.

As in \cite{Brunet2015}, we only consider initial conditions $u_0(x)$
such that
\begin{equation}
u_0(x)=1\text{ for $x\le0$},\qquad u_0(x)\in[0,1)\text{ for $x\ge1$},
\qquad u_0(x+1)\le u_0(x),
\end{equation}
and we introduce, for each $x\ge1$, the time $t_x$ at which $u(x,\cdot)$
reaches 1. It is clear that $x\mapsto t_x$ is an increasing sequences and
it was shown in \cite{Brunet2015} how to obtain the asymptotic expansion
of $t_x$ for large $x$ up to the term $1/\sqrt x$. Pushing the same technique
one step further, we obtain that for an initial condition $u_0(x)$ that
decays fast enough (see below), then
\begin{equation}
t_x = \frac{x}{v_c}+\frac1{\gamma_cv_c}\Big[\frac32\ln x +c
+\frac{d}{\sqrt x} +f \frac{\ln x}{x}+\mathcal O\Big(\frac1x\Big)\Big]
\quad\text{with }d=3\sqrt{\pi\frac{2v_c}{v''(\gamma_c)}}\gamma_c ^{-3/2},
\label{tx}
\end{equation}
where $c$ depends in a fine way on the initial condition and where $f$ is
a complicated expression involving $v_c$, $\gamma_c$, $v''(\gamma_c)$ and
$v'''(\gamma_c)$. One can check
that the $\ln x$ term in \eqref{tx} is valid if $\sum_x  x u_0(x)
e^{\gamma_c x}<\infty$  as in \eqref{stronger},
the $1/\sqrt x$ term
is correct if $\sum_x x^2 u_0(x)e^{\gamma_c x}<\infty$, and the $(\ln
x)/x$ term is correct if $\sum_x x^3 u_0(x)e^{\gamma_c x}<\infty$.
This asymptotic expansion
was in \cite{Brunet2015} up to the $1/\sqrt x$ term.

If one inverts formally \eqref{tx}, one obtains
\begin{equation}
x_t = v_c t -\frac3{2\gamma_c}\ln t + c' - \frac{d'}{\sqrt t}+
f'\frac{\ln t}t+\mathcal O\Big(\frac1t\Big),
\label{xt}
\end{equation}
with
\begin{equation}
d'=\frac{d}{\gamma_c\sqrt{v_c}}=3\sqrt{\pi\frac{2}{v''(\gamma_c)}}\gamma_c
^{-5/2},
\qquad
f'=\frac{9}{4\gamma_c^2v_c}-\frac f{\gamma_c v_c}=
\frac{54-54\ln
2 +3\gamma_c\frac{v'''(\gamma_c)}{v''(\gamma_c)}}{4\gamma_c^4v''(\gamma_c)}.
\label{d'f'}
\end{equation}
(The derivation of the value of $f$ or $f'$ is mechanical and tedious and
of little interest. It is a simple application of the techniques explained
in \cite{Brunet2015} pushed one step further.)

Conjecture~\ref{conj2} relies simply on the idea that the asymptotic
expansion of the $\mu_t\uppar\alpha$ in the Fisher-KPP equation is
also
given by \eqref{xt} with an $\alpha$-dependent constant $c'$, as in
\eqref{Bram}. However, from Conjecture~\ref{conj1}, the coefficient of the
$(\ln t)/t$ term should be independent of $\alpha$.
If one assumes that \eqref{d'f'}
also holds for the Fisher-KPP, one obtains,
with $\gamma_c=1$, $v''(\gamma_c)=2$ and $v'''(\gamma_c)=-6$, that
\begin{equation}
\mu_t\uppar\alpha = 2t-\frac32\ln t +C + W\uppar\alpha
-\frac{3\sqrt{\pi}}{\sqrt t} +\frac98(5-6\ln2)\frac{\ln t}t+\mathcal
O\Big(\frac1t\Big), 
\label{conj}
\end{equation}
where one recognizes in particular the Ebert and van Saarloos term
\cite{evs}.

\medskip

We tried to see if we could see this $(\ln t)/t$ in the numerical
simulations we discussed in Section~\ref{numev}. To do
this, we first subtracted all the known terms in $\mu_t\uppar\alpha$ and
computed
\begin{equation}
\delta_t\uppar\alpha = \mu_t\uppar\alpha -v_c t +\frac3{2\gamma_c}\ln t +
\frac{d'}{\sqrt t}.
\end{equation}
(Of course, we used $v_c$ and $\gamma_c$ given by \eqref{vcgc}. Similarly,
the value used for $d'$ is not $3\sqrt\pi$ but the value given in
\eqref{d'f'}.) We then fitted (using \texttt{gnuplot}) the
$\delta_t\uppar\alpha$ to extract the parameters we needed. Performing this
fit is difficult: we fit against asymptotic expansions, so we need to
consider large times only. On the other hand, if one fits over too narrow
an interval, it is very hard to distinguish between $(\ln t)/t$ and $1/t$.
To overcome these difficulties, it seemed necessary to fit over a large time
interval (to be able to distinguish a $\ln t$ from a constant) and to
include more terms in the expansion to gain in accuracy at smaller times.
To allow the reader to better evaluate our numerical results, we present
results for several fits: we used the following candidates for fitting the data:
\begin{equation}
\begin{gathered}
\text{(a)}\quad \delta_t = C +\frac{f'\ln t + g}t,\qquad
\text{(b)}\quad \delta_t = C +\frac{f'\ln t + g}t+\frac{h\ln
t +i}{t^{3/2}},\\
\text{(c)}\quad \delta_t = C +\frac{f'\ln t + g}t+\frac{h\ln t +i}{t^{3/2}}
+\frac{j\ln t +k}{t^2},
\end{gathered}
\label{fitf}
\end{equation}
over different ranges of $t$.
 The values of $f'$ extracted from
the fits are presented
in Table~\ref{fit}. When using function~(a), these values depend a lot on
the chosen range. This is because the effects of smaller terms in the
expansion is not negligible enough for the values of $t$ that we could reach.
Function~(b) seems to suffer a little bit from this effect, but to a much
lesser extent. Function~(c) leads to a remarkable uniformity of values for
$f'$.
According to \eqref{d'f'}, the value of $f'$ should be 0.948\ldots,
which is in quite  good agreement with the fitted values.
(For the Fisher-KPP, the value for $f'$ in \eqref{conj} is 0.946\ldots).

\def\dat#1#2#3#4#5{\begin{tabular}{c}#1\\#2\\#3\\#4\\#5\end{tabular}}
\begin{table}[!ht]
\centering
\begin{tabular}{|l||c|c|c|}
\hline
& on [100,85000]& on [1000,85000]&on [10000,85000]\\
\hline\hline
With function (a) &
\dat{1.642}{1.639}{1.630}{1.619}{1.501}&
\dat{1.355}{1.302}{1.288}{1.274}{1.177}&
\dat{1.164}{1.131}{1.123}{1.114}{1.060}\\
\hline
With function (b) &
\dat{0.805}{0.896}{0.912}{0.926}{0.979}&
\dat{0.907}{0.928}{0.932}{0.937}{0.958}&
\dat{0.933}{0.938}{0.939}{0.941}{0.947}\\
\hline
With function (c) &
\dat{0.938}{0.945}{0.945}{0.944}{0.935}&
\dat{0.945}{0.944}{0.944}{0.944}{0.943}&
\dat{0.937}{0.936}{0.938}{0.938}{0.935}\\
\hline
\end{tabular}
\caption{The value of $f'$ when fitting the $\delta_t\uppar\alpha$ against
the functions in~\eqref{fitf} over three time ranges. In each cell, the
five values correspond from top to bottom to $\alpha=0.01$, $\alpha=0.3$,
$\alpha=0.5$, $\alpha=0.7$ and $\alpha=0.99$.}
\label{fit}
\end{table}

On Figure~\ref{fig2},
$t(\delta_t\uppar\alpha-C\uppar\alpha)$ is plotted
as a function of $t$ on a log-lin
scale, using for $C\uppar\alpha$ the value obtained from the fit with
function~(c) over $[1000,85000]$. The curves seem to have an
asymptote, which would indicate that
\begin{itemize}
\item The Ebert-van Saarloos correction in $1/\sqrt t$ is indeed the first
vanishing term in the asymptotic expansion of $\mu_t\uppar\alpha$, with the
predicted coefficient. (If the prefactor were wrong, the curves in
Figure~\ref{fig2} would blow up exponentially fast in the log-lin scale.)
\item After the Ebert-van Saarloos correction, the next term in the
asymptotic expansion of $\mu_t\uppar\alpha$
seems indeed to be a $(\ln t)/t$.
\item The prefactor of the $(\ln t)/t$, which is given by the slope of the
asymptote of the curves in Figure~\ref{fig2}, is possibly equal
to the $\alpha$-independent value predicted in~\eqref{d'f'}.
\end{itemize}

\begin{figure}[!ht]
\centering
\includegraphics[width=.6\textwidth]{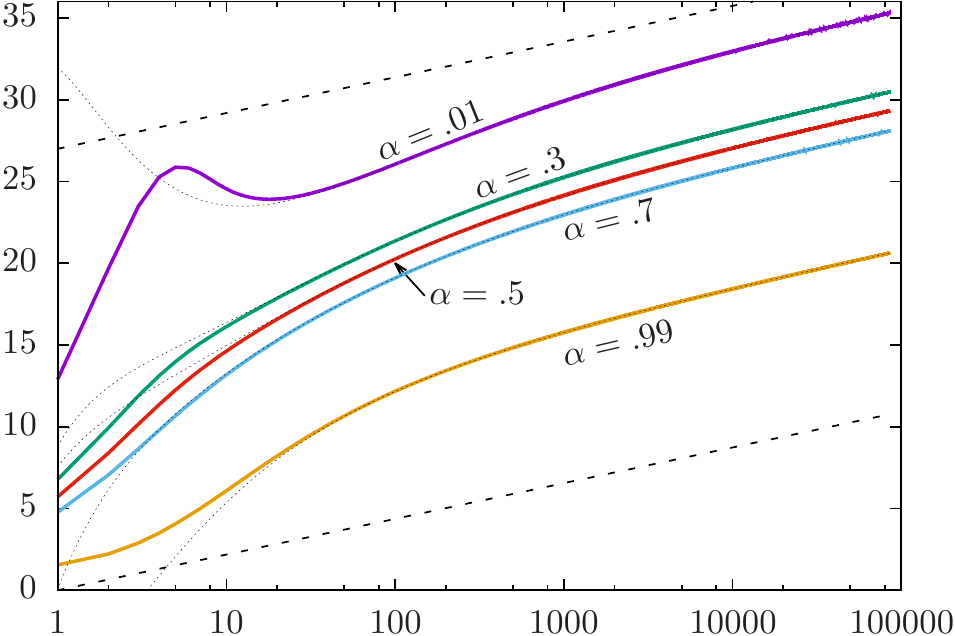}
\caption{$t(\delta_t\uppar\alpha
-C\uppar\alpha)$ as a function of $t$, on a log-lin scale.
The value of $C\uppar\alpha$ was obtained from the fit with function~(c)
over $[1000,85000]$. The small dotted lines show, for each $\alpha$ the
result of the fit.
The two straight
dashed lines are $0.946\ln t + \text{Cste}$. }
\label{fig2}
\end{figure}

\bibliographystyle{abbrv}
\bibliography{bibliography}

\begin{thebibliography}{1}

\bibitem{BBHR2015}
J.~Berestycki, {\'E}.~Brunet, S.~C. Harris, and M.~I. Roberts.
\newblock Vanishing corrections for the position in a linear model of {FKPP}
  fronts.
\newblock arXiv:1510.03329, 2015.

\bibitem{Bramson83}
M.~Bramson.
\newblock Convergence of solutions of the {K}olmogorov equation to travelling
  waves.
\newblock {\em Mem. Amer. Math. Soc.}, 44(285):iv+190, 1983.

\bibitem{Brunet2015}
{\'E}.~Brunet and B.~Derrida.
\newblock An exactly solvable travelling wave equation in the {F}isher--{KPP}
  class.
\newblock {\em Journal of Statistical Physics}, pages 1--20, 2015.

\bibitem{evs}
U.~Ebert and W.~van Saarloos.
\newblock Universal algebraic relaxation of fronts propagating into an unstable
  state and implications for moving boundary approximations.
\newblock {\em Physical review letters}, 80(81):1650, 1998.

\bibitem{Fisher}
R.~A. Fisher.
\newblock The wave of advance of advantageous genes.
\newblock {\em Annals of Eugenics}, 7(4):355--369, 1937.

\bibitem{KPP}
A.~Kolmogorov, I.~Petrovsky, and N.~Piscounov.
\newblock {\'E}tude de l'\'equation de la diffusion avec croissance de la
  quantit\'e de mati\`ere et son application \`a un probl\`eme biologique.
\newblock {\em Bull. Univ. \'Etat Moscou, A}, 1(6):1--25, 1937.

\bibitem{nrr}
L.~Ryzhik, J.~Nolen, and J.-M. Roquejoffre.
\newblock Refined long time asymptotics for the {F}isher-{KPP} equation.
\newblock To appear, 2015.

\bibitem{uchiyama}
K.~Uchiyama.
\newblock The behavior of solutions of some non-linear diffusion equations for
  large time.
\newblock {\em Journal of Mathematics of Kyoto University}, 18(3):453--508,
  1978.

\end{thebibliography}

\end{document}